\documentclass{amsart}

\usepackage{mathrsfs}
\usepackage{amsfonts}

\usepackage{caption}

\usepackage[hyperindex,CJKbookmarks]{hyperref}
\usepackage{cite}
\usepackage{graphicx}

\newcommand{\balg}{\begin{algorithm}}
\newcommand{\ealg}{\end{algorithm}}
\newcommand{\br}{\begin{remark}}
\newcommand{\er}{\end{remark}}
\newcommand{\bex}{\begin{example}}
\newcommand{\eex}{\end{example}}

\newtheorem{theorem}{Theorem}[section]

\newtheorem{proposition}[theorem]{Proposition}

\theoremstyle{definition}

\newtheorem{example}[theorem]{Example}

\newtheorem{algorithm}[theorem]{Algorithm}

\newtheorem{remark}[theorem]{Remark}

\numberwithin{equation}{section}

\begin{document}

\title{The CP-matrix completion problem}


\author{Anwa Zhou}
\address{Department of Mathematics, Shanghai Jiao Tong University,
Shanghai 200240, P.R. China}
\email{congcongyan@sjtu.edu.cn}
\thanks{}

\author{Jinyan Fan}
\address{Department of Mathematics, and MOE-LSC, Shanghai Jiao Tong University,
Shanghai 200240, P.R. China}
\email{jyfan@sjtu.edu.cn}
\thanks{The second author is partially supported by NSFC 11171217.}


\date{}

\dedicatory{}

\begin{abstract}
A symmetric matrix $C$ is completely positive (CP)
if there exists an entrywise nonnegative matrix $B$ such that $C=BB^T$.
The CP-completion problem is to study whether we can assign values
to the missing entries of a partial matrix
(i.e., a matrix having unknown entries) such that
the completed matrix is completely positive.
We propose a semidefinite algorithm for solving general CP-completion problems,
and study its properties.
When all the diagonal entries are given, the algorithm can
give a certificate if a partial matrix is not CP-completable,
and it almost always gives a CP-completion if it is CP-completable.
When diagonal entries are partially given,
similar properties hold. Computational experiments are also presented
to show how CP-completion problems can be solved.
\end{abstract}


\maketitle

\section{Introduction}

A matrix is partial if some of its entries are missing.
The matrix completion problem is to study whether
we can assign values to the missing entries of a partial matrix
such that the completed matrix satisfies certain properties,
e.g., it is positive semidefinite or an Euclidean distance matrix.
This problem has wide applications, as shown in \cite{HuangLP,CandesR,Laurent01,Al-HomidanW,NegahbanW}.
We refer to Laurent's survey \cite{Laurent} and the references therein.
Interesting applications include the Netflix problem \cite{Netflix,JamesS},
global positioning \cite{CandesP}, multi-task
learning \cite{AbernethyEV06,AbernethyEV09,ArgyriouEP}, etc.
The motivation of this paper is to study whether or not a partial matrix
can be completed to a matrix that is completely positive.

A real $n\times n$ symmetric matrix $C$ is completely
positive (CP) if there exist nonnegative vectors $u_1,\cdots,u_m$
in $\mathbb{R}^n$ such that
\begin{equation}
C = u_1 u_1^T +\cdots +u_m u_m^T,\label{CPd}
\end{equation}
where $m$ is called the  length of the decomposition
(\ref{CPd}). The smallest $m$ in the above is called the
 CP-rank  of $C$. If $C$ is complete positive, we call (\ref{CPd}) a
 CP-decomposition of $C$. Clearly, a matrix $C$ is completely positive if and only if
$C = B B^T$ for an entrywise nonnegative matrix $B$.
Hence, a CP-matrix is not only positive semidefinite but also nonnegative entrywise.

CP-matrices have wide applications in mixed binary quadratic programming \cite{Burer}, approximating
stability numbers \cite{deKlerk}, max clique problems \cite{PardalosR,ZilinskasD},
single quadratic constraint problems \cite{Preisig},
standard quadratic optimization problems~\cite{Bomze},
and general quadratic programming \cite{Quist}.
Some NP-hard problems can be formulated as linear optimization
problems over the cone of CP-matrices (cf.~\cite{Dur,Hiriart-Urruty,Murty}).
We refer to \cite{BermanN,Burer1,BermanS,BomzeSU,Bundfuss,Dickinson,DickinsonD,DickinsonL,Bomze1}
for the work in this field.
These important applications motivate us to
study the so-called CP-completion problem. Let
\[
E=\{(i_k,j_k) \mid 1\leq i_k \leq j_k \leq n, k=1,\cdots,l\}.
\]
be an index set of pairs.
A partial symmetric matrix $A$ is called an $E$-matrix
if its entries $A_{ij}$ are given for all $(i,j)\in E$,
while $A_{ij}$ are missing for $(i,j)\not\in E$.
The CP-completion problem is to study whether we can assign values
to the missing entries of an $E$-matrix such that
the completed matrix is completely positive.
If such an assignment exists, we say that the $E$-matrix is CP-completable;
otherwise, we say that it is not CP-completable
(or {non-CP-completable}).

A symmetric matrix can be identified by a vector that
consists of its upper triangular entries.
Similarly, an $E$-matrix $A$ can be identified as a vector
\[\mathbf{a} \in \mathbb{R}^E,\]
such that $\mathbf{a}_{ij} = A_{ij}$ for all $(i,j) \in E$.
(The symbol $\mathbb{R}^E$ stands for the space of all real vectors
indexed by pairs in $E$.)
For a matrix $F$, we denote by $F|_E$
the vector in $\mathbb{R}^E$ whose $(i,j)$-entry is $F_{ij}$, for all $(i,j) \in E$.
Clearly, an $E$-matrix $A$ is {CP-completable} if and only if there exists
a CP-matrix $C$ such that $\mathbf{a} = C|_E$.

For example, consider the $E$-matrix $A$ given as  (\cite[Example 2.23]{BermanN}):
\begin{equation}\label{ex223}
\left[\begin{array}{cccc}
  2& 3 &0 &\ast\\
  3& 6 &3 &0\\
  0& 3 &6 &3\\
  \ast &0 &3 &2
\end{array}\right],
\end{equation}
where $\ast$ means that the entry there is missing, throughout the paper.
The index set $E$ is
\[
\{(1,1),(1,2),(1,3),(2,2),(2,3),(2,4),(3,3),(3,4),(4,4)\},
\]
and the identifying vector $\mathbf{a}$ of $A$ is
\[
(2,3,0,6,3,0,6,3,2).
\]
If we assign the missing entry $A_{14}$ a value, say, $t$,
the determinant of $A$ is $-27(t +1)$.
So, $A$ can not be positive semidefinite for any $t>-1$.
This implies that the $E$-matrix $A$ is not CP-completable.

For another example, consider the $E$-matrix $A$ given as:
\begin{equation}\label{ex13}
\left[\begin{array}{ccc}
  1& 1 &1\\
  1& 1 &1\\
  1& 1 &\ast
\end{array}\right].
\end{equation}
The index set $E$ is $\{(1,1),(1,2),(1,3),(2,2),(2,3)\}$ and
the identifying vector
$\mathbf{a}$ is $(1,1,1,1,1)$. Since
\[
\left[\begin{array}{ccc}
  1& 1 &1\\
  1& 1 &1\\
  1& 1 &1
\end{array}\right]=\left(\begin{array}{c}
  1\\
  1\\
  1
\end{array}\right)\left(\begin{array}{c}
  1\\
  1\\
  1
\end{array}\right)^T
\]
is completely positive, we know that
this $E$-matrix is CP-completable.

Note that if a diagonal entry of a CP-matrix is zero,
then all the entries in its row or column
are zeros. Without loss of generality,
we can assume that all the given entries of an $E$-matrix
are nonnegative and all its given diagonal entries are strictly positive.
Otherwise, it can be reduced to a smaller $E^\prime$-matrix
with some index set $E^\prime$.

An $E$-matrix $A$ is called a partial CP-matrix if
every principal submatrix of $A$,
whose entry indices are all from $E$, is completely positive.
When all the diagonal entries are given,
the specification graph of an $E$-matrix is defined as the graph
whose vertex set is $\{1,2,\ldots,n\}$ and
whose edge set is $\{ (i,j) \in E: \, i\neq j \}$.
A symmetric partial matrix is called a matrix realization of a graph $G$
if its specification graph is $G$.
It is shown in \cite{DrewJ, BermanN} that
a partial CP-matrix realization of a connected graph
$G$ is CP-completable if and only if $G$ is a block-clique graph.
Under some conditions, a partial CP-matrix,
whose specification graph $G$ contains cycles, is CP-completable
if and only if all the blocks of $G$ are cliques or cycles \cite{DrewJKM}.
These results assume that all the diagonal entries are given and the
specification graphs satisfy certain combinatorial properties.

In this paper, we study general CP-completion problems in a unified framework.
If an $E$-matrix is not CP-completable, how
can we get a certificate for this? If it is CP-completable, how can we
get a CP-completion and a CP-decomposition for the completed matrix?
To the best knowledge  of the authors,
there exists few work on solving general CP-completion problems.

The paper is organized as follows. In Section 2,
we give a semidefinite algorithm for solving general CP-completion problems,
after an introduction of truncated moment problems.
Its basic properties are also studied.
In Section 3, we study properties of CP-completion problems
when some diagonal entries are missing.
Computational results are given in Section 4.
Finally, we conclude the paper in Section 5
with some applications and discussions about future work.

\section{A semidefinite algorithm for CP-completion}

Recently, Nie \cite{Nie} proposed a semidefinite algorithm for
solving $\mathcal {A}$-truncated $K$-moment problems ($\mathcal {A}$-TKMPs),
which are generalizations of classical
truncated $K$-moment problems (cf.~\cite{Helton}).
In this section,  we show how to formulate CP-completion problems
in the framework of $\mathcal{A}$-TKMP,
and then propose a semidefinite algorithm to solve them.

\subsection{CP-completion as $\mathcal{E}$-T$\Delta$MP}

First, we characterize when an $E$-matrix is CP-completable. Let
\begin{equation}
\Delta=\{x\in \mathbb{R}^n :\,
x_1+\cdots +x_n -1= 0, x_1 \geq 0, \ldots, x_n \geq 0\}  \label{KE}
\end{equation}
be  the  standard simplex in $\mathbb{R}^n$. For convenience, denote the polynomials:
\[
h(x):= x_1+\cdots +x_n -1, \, g_1(x): =x_1, \, \ldots,  g_n(x): = x_n .
\]
Let $A$ be an $E$-matrix with the identifying vector $\mathbf{a} \in \mathbb{R}^E$.
We have seen that $A$ is CP-completable if and only if
$\mathbf{a} = C|_E$ for some CP-matrix $C$.
Note that every nonnegative vector is a multiple
of a vector in the simplex $\Delta$.
So, in view of (\ref{CPd}),
$A$ is CP-completable if and only if
there exist vectors
$v_1,\cdots,v_m \in \Delta$ and $\rho_1,\cdots,\rho_m >0$ such that
 \begin{equation}\label{ECPe}
C =\rho_1 v_1 v_1^T +\cdots +\rho_m v_m v_m^T \quad \text{and} \quad \mathbf{a} = C|_E.
 \end{equation}
Let $\mathbb{N}$ be the set of nonnegative integers.
For $\alpha = (\alpha_1,\ldots, \alpha_n) \in \mathbb{N}^n$,
denote $|\alpha| := \alpha_1+\cdots+\alpha_n$. Let
$
\mathbb{N}_d^n := \{ \alpha \in \mathbb{N}^n: \, |\alpha| \leq n\}.
$
Denote
\begin{equation}
\label{AE} \mathcal {E} := \{\alpha
\in \mathbb{N}^n:\, \alpha = e_i + e_j, (i,j)\in E\},
\end{equation}
where $\textbf{e}_i$ is the $i$-th unit vector.
For instance, when $n=3$ and $E=\{(1,2),(2,2),(2,3)\}$, then
$\mathcal {E}=\{(1,1,0),(0,2,0),(0,1,1)\}$. The degree
$deg(\mathcal {E}) := \max\{|\alpha|: \alpha \in \mathcal {E}\}$ is two
for all $E$. Since there is a one-to-one correspondence
between $\mathcal {E}$ and $E$, we can also
index the identifying vectors $\mathbf{a} \in \mathbb{R}^E$
of $E$-matrices as
 \begin{equation}
\mathbf{a}=\label{yalphaE}(\mathbf{a}_{\alpha})_{\alpha\in \mathcal {E}}
\in \mathbb{R}^{\mathcal {E}}, \qquad
\mathbf{a}_{\alpha} =\mathbf{a}_{ij} \quad
\mbox{ if } \quad \alpha = e_i + e_j.
\end{equation}
($\mathbb{R}^{\mathcal {E}}$ denotes the
space of real vectors indexed by elements in $\mathcal {E}$.)
We call such $\mathbf{a}$
an {$\mathcal {E}$-truncated moment sequence} ($\mathcal {E} $-tms) (cf. \cite{Nie}).

The {$\mathcal {E}$-truncated $\Delta$-moment problem}
($\mathcal {E}$-T$\Delta$MP) studies whether or not a given {$\mathcal {E}$-tms $\mathbf{a}$ admits a
$\Delta$-measure} $\mu$, i.e., a nonnegative Borel measure $\mu$ supported
in $\Delta$ such that
\[
\mathbf{a}_{\alpha}= \int_\Delta x^{\alpha} d \mu \quad
\forall \, \alpha \in  \mathcal {E},
\]
where $x^{\alpha} := x^{\alpha_1}_1 \cdots x^{\alpha_n}_n$.
A measure $\mu$ satisfying the above is called a
{$\Delta$-representing measure} for $\mathbf{a}$. A measure is called {finitely
atomic} if its support is a finite set, and is called $m$-atomic if its
support consists of at most $m$ distinct points.
We refer to \cite{Nie} for representing measures
of truncated moment sequences.

Hence, by (\ref{ECPe}), an $E$-matrix $A$, with the identifying vector
$\mathbf{a} \in \mathbb{R}^{\mathcal {E}}$, is CP-completable if and only if $\mathbf{a}$ admits
an $m$-atomic $\Delta$-measure,  i.e.,
\begin{equation}
\label{ECPee} \mathbf{a}=\rho_1 [v_1]_{\mathcal {E}} +\cdots +\rho_m [v_m]_{\mathcal
{E}},
\end{equation}
with each $v_i \in \Delta$ and $\rho_i>0$.
In the above, we denote
\[
[v]_{\mathcal {E}} :=(v^{\alpha})_{\alpha \in
{\mathcal {E}}}.
\]
In other words,
CP-completion problems are equivalent to $\mathcal {E}$-T$\Delta$MPs with
$\mathcal {E}$ and $\Delta$ given by (\ref{AE}) and (\ref{KE}) respectively.

\subsection{A semidefinite algorithm}

We present a semidefinite algorithm for solving CP-completion problems,
by formulating them as $\mathcal {E}$-T$\Delta$MPs.
To describe it, we need to introduce localizing matrices. Denote
\[
\mathbb{R}[x]_{\mathcal {E}}:= \mbox{span}\{x^{\alpha}: \alpha\in \mathcal {E}\}.
\]
We say that $\mathbb{R}[x]_{\mathcal {E}}$ is
{$\Delta$-full} if there exists a polynomial $p \in \mathbb{R}[x]_{\mathcal {E}}$ such
that $p|_{\Delta} >0$ (cf. \cite{Nie4}). Let
$\mathbb{R}[x]_{d}:=\mbox{span}\{x^{\alpha}: \alpha\in \mathbb{N}^n_{d}\}$.
An $\mathcal {E}$-tms $y \in \mathbb{R}^{\mathcal {E}}$
defines an $\mathcal {E}$-Riesz function $\mathscr{L}_{y}$ acting on
$\mathbb{R}[x]_{\mathcal {E}}$ as
 \begin{equation}\label{Ly} \mathscr{L}_{y}
(\sum_{\alpha\in \mathcal {E}}p_{\alpha} x^{\alpha}):=
\sum_{\alpha\in \mathcal {E}}p_{\alpha} y_{\alpha}.
 \end{equation}
For $z \in \mathbb{R}^{\mathbb{N}^n_{2k}}$ and $q
\in \mathbb{R}[x]_{2k}$, the $k$-th localizing matrix  of
$q$ generated by $z$ is the symmetric
matrix $L^{(k)}_q (z)$ satisfying
 \begin{equation}\label{Lzqp2}
\mathscr{L}_z (q p^2)= vec(p)^T (L^{(k)}_q (z)) vec(p) \quad \forall p\in
\mathbb{R}[x]_{k-\lceil deg(q)/2 \rceil}.
 \end{equation}
In the above,
$vec(p)$ denotes the coefficient vector of $p$
in the graded lexicographical ordering, and $\lceil t\rceil$
denotes the smallest integer that is not smaller than $t$.
In particular, when $q=1$, $L^{(k)}_1 (z)$ is called a  $k$-th
order moment matrix  and denoted as $M_k(z)$.
We refer to \cite{Nie, Nie2} for more details
about localizing and moment matrices.

Let $g_0(x) := 1$ and $g_{n+1}(x) := 1 - \|x\|^2_2$. Since
$\Delta \subseteq B(0,1):=\{x\in \mathbb{R}^n : \|x\|^2_{2} \leq 1\}$,
we can also describe $\Delta$ equivalently as
\[
\Delta =\{x\in \mathbb{R}^n: h(x)= 0, g(x) \geq 0\},
\]
where $g(x):= (g_0(x), g_1(x),\ldots,g_n(x), g_{n+1}(x))$.
For instance, when $n = 2$ and $k = 2$,
the second order localizing matrices of the above polynomials are:
{\small
$$
 L^{(2)}_{x_1\!+\! x_2 \!-\!1} (z)=\left[\begin{array}{ccc}
  z_{(1,0)}\!+\!z_{(0,1)}\!-\!z_{(0,0)} &z_{(2,0)}\!+\!z_{(1,1)}\!-\!z_{(1,0)} &z_{(1,1)}\!+\!z_{(0,2)}\!-\!z_{(0,1)}\\
  z_{(2,0)}\!+\!z_{(1,1)}\!-\!z_{(1,0)} &z_{(3,0)}\!+\!z_{(2,1)}\!-\!z_{(2,0)} &z_{(2,1)}\!+\!z_{(1,2)}\!-\!z_{(1,1)} \\
  z_{(1,1)}\!+\!z_{(0,2)}\!-\!z_{(0,1)} &z_{(2,1)}\!+\!z_{(1,2)}\!-\!z_{(1,1)} &z_{(1,2)}\!+\!z_{(0,3)}\!-\!z_{(0,2)}
\end{array}\right],
 $$
  $$
 M_2(z) :=  L^{(2)}_{1} (z)= \left[\begin{array}{cccccc}
  z_{(0,0)} &z_{(1,0)} &z_{(0,1)} &z_{(2,0)} &z_{(1,1)} &z_{(0,2)}\\
  z_{(1,0)} &z_{(2,0)} &z_{(1,1)} &z_{(3,0)} &z_{(2,1)} &z_{(1,2)}\\
  z_{(0,1)} &z_{(1,1)} &z_{(0,2)} &z_{(2,1)} &z_{(1,2)} &z_{(0,3)}\\
  z_{(2,0)} &z_{(3,0)} &z_{(2,1)} &z_{(4,0)} &z_{(3,1)} &z_{(2,2)}\\
  z_{(1,1)} &z_{(2,1)} &z_{(1,2)} &z_{(3,1)} &z_{(2,2)} &z_{(1,3)}\\
  z_{(0,2)} &z_{(1,2)} &z_{(0,3)} &z_{(2,2)} &z_{(1,3)} &z_{(0,4)}
\end{array}\right],
 $$
  $$
 L^{(2)}_{x_1} (z)=\left[\begin{array}{ccc}
  z_{(1,0)} &z_{(2,0)} &z_{(1,1)}\\
  z_{(2,0)} &z_{(3,0)} &z_{(2,1)} \\
  z_{(1,1)} &z_{(2,1)} &z_{(1,2)}
\end{array}\right], \qquad
 L^{(2)}_{x_2} (z)= \left[\begin{array}{ccc}
  z_{(0,1)} &z_{(1,1)} &z_{(0,2)}\\
  z_{(1,1)} &z_{(2,1)} &z_{(1,2)} \\
  z_{(0,2)} &z_{(1,2)} &z_{(0,3)}
\end{array}\right],
$$
$$
 L^{(2)}_{1\!-\!x_1^2\!-\!x_2^2} (z)= \left[\begin{array}{ccc}
 z_{(0,0)}\!-\!z_{(2,0)}\!-\!z_{(0,2)} &z_{(1,0)}\!-\!z_{(3,0)}\!-\!z_{(1,2)} &z_{(0,1)}\!-\!z_{(2,1)}\!-\!z_{(0,3)}\\
 z_{(1,0)}\!-\!z_{(3,0)}\!-\!z_{(1,2)} &z_{(2,0)}\!-\!z_{(4,0)}\!-\!z_{(2,2)} &z_{(1,1)}\!-\!z_{(3,1)}\!-\!z_{(1,3)} \\
 z_{(0,1)}\!-\!z_{(2,1)}\!-\!z_{(0,3)} &z_{(1,1)}\!-\!z_{(3,1)}\!-\!z_{(1,3)} &z_{(0,2)}\!-\!z_{(2,2)}\!-\!z_{(0,4)}
\end{array}\right].
$$
}
As shown in \cite{Nie}, a necessary condition for $z \in \mathbb{R}^{\mathbb{N}^n_{2k}}$
to admit a $\Delta$-measure is
\begin{equation}
\label{SDPC}
  L^{(k)}_{h} (z) = 0, \quad \mbox{and}\quad L^{(k)}_{g_j} (z) \succeq
0, \quad j=0,1,\ldots,n+1.
\end{equation}
(In the above, $X \succeq 0$ means that $X$ is positive semidefinite.)
If, in addition to (\ref{SDPC}), $z$ satisfies the {rank condition}
\begin{equation}
\label{RC}
\text{rank} M_{k-1}(z) =\text{rank} M_{k} (z),
\end{equation}
then $z$ admits a unique
$\Delta$-measure, which is $\text{rank} M_k(z)$-atomic
(cf.~Curto and Fialkow~\cite{CurtoF}).
We say that $z$ is {flat} if (\ref{SDPC}) and (\ref{RC}) are both satisfied.

Given two tms' $y \in \mathbb{R}^{\mathbb{N}^n_{d}}$ and $z \in
\mathbb{R}^{\mathbb{N}^n_{e}}$, we say $z$ is an  extension  of $y$, if $d\leq e$ and
$y_{\alpha} = z_{\alpha}$ for all $\alpha \in \mathbb{N}^n_{d}$. We denote
by $z|_{\mathcal {E}}$ the subvector of $z$, whose entries are indexed by
$\alpha \in \mathcal {E}$. For convenience, we denote by $z|_{d}$
the subvector $z |_{\mathbb{N}^n_{d}}$.
If $z$ is flat and extends $y$, we say $z$ is a {flat
extension} of $y$. Note that an $\mathcal {E}$-tms
$\mathbf{a} \in \mathbb{R}^{\mathcal{E}}$ admits a $\Delta$-measure
if and only if it is extendable to a flat tms $z \in \mathbb{R}^{\mathbb{N}^n_{2k}}$
for some $k$ (cf.~\cite{Nie}).
By (\ref{ECPee}), determining whether an $E$-matrix $A$
is CP-completable or not is equivalent to
investigating whether $\mathbf{a}$ has a flat extension or not.

Let $d >2$ be an even integer. Choose a polynomial
$R \in \mathbb{R}[x]_d$ and write it as
\[
R(x)  = \sum\limits_{\alpha\in \mathbb{N}^n_{d}}R_{\alpha} x^{\alpha}.
\]
Consider the linear optimization problem
\begin{equation}
\label{MP}
  \begin{array}{ll}
  \min\limits_{z}  & \sum\limits_{\alpha\in \mathbb{N}^n_{d}}R_{\alpha} z_{\alpha}\\
  \mbox{s.t.} & z|_{\mathcal {E}} = \mathbf{a}, z \in \Upsilon_d (\Delta),
 \end{array}
\end{equation}
where $\Upsilon_d (\Delta)$ is the set
of all tms' $z \in \mathbb{R}^{\mathbb{N}_d^n}$ admitting $\Delta$-measures.
Note that $\Delta$ is a compact set.
It is shown in \cite{Nie} that if $\mathbb{R}[x]_{\mathcal {E}}$ is $\Delta$-full,
then the feasible set of (\ref{MP}) is compact convex and (\ref{MP}) has a minimizer for all $R$.
If $\mathbb{R}[x]_{\mathcal {E}}$ is not $\Delta$-full, we
need to choose $R$ which is positive definite on $\Delta$,
to guarantee that (\ref{MP}) has a minimizer.
Therefore, to get a minimizer of (\ref{MP}), it is enough to solve (\ref{MP}) for a
generic positive definite $R$, no matter whether $\mathbb{R}[x]_{\mathcal {E}}$ is $\Delta$-full
or not. For this reason, we choose  $R \in \Sigma_{n,d}$, where $\Sigma_{n,d}$ is
the set of all sum of squares polynomials in $n$ variables with degree $d$.
Since $\Upsilon_d (\Delta)$ is typically quite difficult to
describe, we relax it by the cone
\begin{equation}
\label{EIk}
  \Gamma_{k} (h,g) := \left\{z \in \mathbb{R}^{\mathbb{N}^n_{2k}}
  \mid  L^{(k)}_{h} (z) = 0, L^{(k)}_{g_j} (z) \succeq 0,
  j=0,1,\ldots,n+1
  \right\},
\end{equation}
with $k \geq d/2$ an integer. The {$k$-th order semidefinite
relaxation} of (\ref{MP}) is
\begin{equation}
\label{SDPR}
(SDR)_k: \qquad \\
  \left\{\begin{array}{ll}
  \min\limits_{z}  & \sum\limits_{\alpha\in \mathbb{N}^n_d }R_{\alpha} z_{\alpha}\\
  \mbox{s.t.} & z|_{\mathcal {E}} = \mathbf{a}, z \in \Gamma_k(h,g).
 \end{array} \right.
\end{equation}
Based on solving the hierarchy of (\ref{SDPR}),
our semidefinite algorithm for solving CP-completion problems
is as follows.

\balg\label{Algorithm}  A semidefinite algorithm for
solving CP-completion problems.

\textbf{Step 0:} Choose a generic $R \in \Sigma_{n,d}$, and let $k := d/2$.

\textbf{Step 1:} Solve (\ref{SDPR}). If (\ref{SDPR}) is infeasible, then
$\mathbf{a}$ doesn't admit a $\Delta$-measure,
i.e., $A$ is not CP-completable, and
stop. Otherwise, compute a minimizer $z^{*,k}$. Let $t :=1$.

\textbf{Step 2:} Let $w := z^{*,k}|_{2t}$. If the rank condition (\ref{RC})
is not satisfied, go to Step 4.

\textbf{Step 3:} Compute the finitely atomic measure $\mu$ admitted by $w$:
$$\mu = \lambda_1 \delta(u_1) + \cdots + \lambda_m \delta(u_m),$$
where $m = \text{rank} M_t (w)$, $u_i \in \Delta$, $\lambda_i >
0$, and $\delta(u_i)$ is the Dirac
measure supported on the point $u_i$ $(i=1,\cdots,m)$. Stop.

\textbf{Step 4:} If $t < k$, set $t := t+1$ and go to Step 2; otherwise, set
$k := k+1$ and go to Step 1. \ealg

\br\label{remark11}\upshape Algorithm \ref{Algorithm} is a specialization of
Algorithm~4.2 in \cite{Nie} to CP-completion.
Denote $[x]_{d} := (x^{\alpha})_{\alpha \in \mathbb{N}^n_{d}}$. We choose $R = [x]^T_{d/2}
J^T J [x]_{d/2}$ in (\ref{SDPR}), where $J$ is a random square matrix
obeying Gaussian distribution. We check the rank condition
(\ref{RC}) numerically with the help of singular value
decompositions~\cite{Golub}. The rank of a matrix is evaluated as the
number of its singular values that are greater than or equal to
$10^{-6}$. We use the method in \cite{HenrionJ} to get a $m$-atomic
$\Delta$-measure for $w$.\er

\subsection{Some properties of Algorithm 2.1}

We first show some basic properties of Algorithm \ref{Algorithm}
which are from \cite[Section 5]{Nie}.
\begin{theorem}
\label{Algorithmresults}
Algorithm \ref{Algorithm} has the following properties:
\begin{itemize}

\item [1)] If (\ref{SDPR}) is infeasible for some $k$,
then $\mathbf{a}$ admits no $\Delta$-measures and the
corresponding $E$-matrix $A$ is not CP-completable.

\item [2)] If the $E$-matrix $A$ is not CP-completable
and $\mathbb{R}[x]_{\mathcal {E}}$ is $\Delta$-full,
then (\ref{SDPR}) is infeasible for all $k$ big enough.

\item [3)] If the $E$-matrix $A$ is CP-completable, then
for almost all generated $R$,
we  can  asymptotically get a flat extension of $\mathbf{a}$
by solving the hierarchy of (\ref{SDPR}).
This gives a CP-completion of $A$.

\end{itemize}
\end{theorem}

\br\label{finiteconvergence}
Under some general conditions, which is almost
sufficient and necessary, we can get a flat extension of
$\mathbf{a}$ by solving the hierarchy of (\ref{SDPR}),
within finitely many step (cf.~\cite{Nie}). This always happens in our
numerical experiments.
So, if an $E$-matrix $A$ with the identifying vector
$\mathbf{a} \in \mathbb{R}^{\mathcal {E}}$
is CP-completable, then we can asymptotically get a flat extension of
$\mathbf{a}$ for almost all $R\in \Sigma_{n,d}$ by Algorithm \ref{Algorithm}.
Moreover, it can often be obtained within finitely many steps.
After getting a flat extension of $\mathbf{a}$, we can get a $m$-atomic
$\Delta$-measure for $\mathbf{a}$, which then produces
a CP-completion of $A$, as well as a CP-decomposition.

\er

When $\mathbb{R}[x]_{\mathcal {E}}$ is $\Delta$-full,
Algorithm \ref{Algorithm} can give a certificate for the non-CP-completability. However,
if it is not $\Delta$-full and $A$ is not CP-completable,
it is not clear whether there exists a $k$ such that (\ref{SDPR}) is infeasible.
This is an open question, to the best knowledge of the authors.
We now characterize when
$\mathbb{R}[x]_{\mathcal {E}}$ is $\Delta$-full.

\begin{proposition}
\label{Kfull}
Suppose $E=\{(i_k,j_k) \mid 1\leq i_k \leq j_k \leq n,
k=1,\cdots,l\}$. Then, $\mathbb{R}[x]_{\mathcal {E}}$ is
$\Delta$-full if and only if $\{(i,i):\, 1\leq i\leq
n\}\subseteq E$.
\end{proposition}

\begin{proof}
We first prove the sufficient condition. If $\{(i,i),1\leq i\leq
n\}\subseteq E$, then
$\{(2,0,\cdots,0),(0,2,\cdots,0),\cdots,(0,0,\cdots,2)\}$ $\subseteq
\mathcal {E}$, so we have $x^2_i \in
\mathbb{R}[x]_{\mathcal {E}}$ for all $1\leq i\leq
n$. Hence for any $x\in \Delta$, there exists a polynomial
$p(x)=\sum^n_{i=1} x_i^2 \in \mathbb{R}[x]_{\mathcal
{E}}$ such that $p(x)>0$. Thus
$\mathbb{R}[x]_{\mathcal {E}}$ is $\Delta$-full.

We prove the necessary condition by contradiction. Suppose there exists some $i_0 \in
\{1,\cdots,n\}$ such that $(i_0,i_0)\notin E$.
For any polynomial $p(x)\in \mathbb{R}[x]_{\mathcal
{E}}$, $p(x)$ is a linear combination of all the monomials of degree 2 except
$x_{i_0}^2$. Let $c=(0,\cdots,0,1_{i_0},0,\cdots,0)^T \in
\Delta$ be a constant vector, then $p(c)=0$ holds for all
$p(x)\in \mathbb{R}[x]_{\mathcal {E}}$. Hence,
there does not exit any polynomial $p(x)\in \mathbb{R}[x]_{\mathcal {E}}$ such that $p(x)|_{\Delta}>0$. This
contradicts the fact that $\mathbb{R}[x]_{\mathcal
{E}}$ is $\Delta$-full. The proof is completed. \end{proof}

\br\label{remarka2}\upshape Proposition \ref{Kfull}
shows that $\mathbb{R}[x]_{\mathcal {E}}$ is
$\Delta$-full if and only if all the diagonal entries
are given. In such case,
Algorithm \ref{Algorithm} can determine whether
an $E$-matrix $A$ can be completed to a
CP-matrix or not. If $A$ is CP-completable, Algorithm \ref{Algorithm}
can give a CP-completion with a CP-decomposition.
If $A$ is not CP-completable, then
it can give a certificate, i.e., (\ref{SDPR}) is infeasible for some $k$.

\er

\section{CP-completion with missing diagonal entries}

When all the diagonal entries are given,
which is equivalent to that $\mathbb{R}[x]_{\mathcal{E}}$ is $\Delta$-full,
the properties of Algorithm~\ref{Algorithm}
are summarized in Theorem~\ref{Algorithmresults} and Remark~\ref{remarka2}.
In this section, we study properties of CP-completion problems
when some diagonal entries are missing.

\subsection{All diagonal entries are missing}

Consider $E$-matrices with
all the diagonal entries missing,
i.e., $(i,i) \not\in E$ for all $i$.
In such case, CP-completion is relatively simple, as shown below.

\begin{proposition}
\label{noknow}
Let $A$ be an $E$-matrix whose entries are all nonnegative.
If all the diagonal entries of $A$ are missing, then $A$
has a CP-completion.
\end{proposition}

\begin{proof}
The matrix
\begin{equation}\label{3.1}
C =\sum_{(i,j)\in E} {A_{ij} (\textbf{e}_i+\textbf{e}_j) (\textbf{e}_i+\textbf{e}_j)^T}
\end{equation}
is clearly completely positive, because each $A_{ij} \geq 0$.
It is easy to check that $C$ is a CP-completion of $A$.
\end{proof}

\br\label{remark2}\upshape
By Proposition \ref{noknow},
every nonnegative $E$-matrix is CP completable,
when all diagonal entries are missing.  In such case,
Algorithm \ref{Algorithm} typically gives a CP-decomposition
whose length is much smaller than the length in the proof
of Proposition \ref{noknow},
which is the cardinality of $E$.
This is an advantage of Algorithm \ref{Algorithm}.
\er

\subsection{Diagonal entries are partially missing}

We consider $E$-matrices whose diagonal entries
are not all missing, i.e.,
the set $E$ contains at least one but not all
of $(1,1), \ldots, (n,n)$.
Suppose the diagonal entry indices in $E$
are $(i_1,i_1), \ldots, (i_r, i_r)$. Let
\[
\hat{E} = \{ (i,j) \in E:\,
i,  j \in \{i_1,\ldots, i_r\} \}.
\]

Let $A$ be an $E$-matrix.
An $\hat{E}$-matrix $P$ is called
the maximum principal submatrix of $A$ if
$P_{ij} = A_{ij}$ for all $(i,j) \in \hat{E}$.
If $P$ is CP-completable, we say that $A$ is partially CP-completable.
Clearly, if $A$ is CP-completable,
then $P$ is also CP-completable.
This immediately leads to the following proposition.

\begin{proposition}
\label{partknow}
If the maximum principal submatrix of an $E$-matrix $A$
is not CP-completable, then $A$ is not CP-completable.
\end{proposition}

\br\label{remark3}\upshape
The converse of Proposition \ref{partknow} is
not necessarily true. For example, consider the $E$-matrix $A$ given as
\begin{equation}\label{cex}
\left[\begin{array}{ccc}
  1& 1 &2\\
  1& 1 &3\\
  2& 3 &\ast
\end{array}\right].
\end{equation}
The determinant of $A$ is always $-1$,
no matter what the $(3,3)$-entry is.
So, it can not be positive semidefinite,
and hence $A$ is not CP-completable.
However, its maximum principal submatrix $P$ is
completely positive, so $A$ is partially CP-completable.
\er

Though the $E$-matrix $A$ in (\ref{cex}) is not CP-completable,
we can show that there exists
a sequence $\{A_k\}$ of CP-completable $E$-matrices
such that their
identifying vectors converge to the one of $A$.

\begin{theorem}
\label{convergence}
Suppose the maximum principal submatrix of an $E$-matrix $A$,
with the identifying vector $\mathbf{a} \in \mathbb{R}^{\mathcal {E}}$,
is CP-completable.
Then there exists a sequence of
CP-completable $E$-matrices $\{ A_k \}$,
 whose identifying vectors converge to $\mathbf{a}$.
\end{theorem}

\begin{proof}
If all the diagonal entries are given,
the theorem is clearly true.

First, we assume exactly one diagonal entry is missing.
Without loss of generality, we assume
$A$ is given in the following form
\begin{equation}\label{ConverExa}
\left[\begin{array}{cccc}
  & & &A_{1,n}\\
  & A'& &\vdots\\
  & & & A_{n-1,n}\\
  A_{n,1}&\cdots & A_{n,n-1}&\ast
\end{array}\right],
\end{equation}
where $A'$ is the maximum principal submatrix of $A$.
If some of the entries $A_{i,n}, A_{n,i}$ ($i=1,\ldots,n-1$) are missing,
we assign the constant value $1$ to them.
The matrix $A'$ is CP-completable, by assumption,
and all its diagonal entries are given.
Consider the following sequence of $E$-matrices:
\begin{equation}
\label{approximation}
\left[\begin{array}{cccc}
  & & &A_{1,n}\\
  & A'+  \varepsilon_k I_{n-1}& &\vdots\\
  & & & A_{n-1,n}\\
  A_{n,1}&\cdots & A_{n,n-1}&\ast
\end{array}\right],\quad k=1,2,\cdots,
\end{equation}
where $I_{n-1}$ is the identity matrix of order $n-1$
and $0 < \varepsilon_k \to 0$ as $k \to \infty$.
Let $\overline{A'}$ be a CP-completion of $A'$, and
$$
A_k = \left[\begin{array}{cc}
  \overline{A'}& \textbf{0}\\
  \textbf{0}^T& 1
\end{array}\right]
+ \sum\limits_{1\leq i\leq n-1} ( \sqrt{\varepsilon_k} \textbf{e}_i
+ \sqrt{\varepsilon_k}^{-1} A_{i,n}  \textbf{e}_{n})
( \sqrt{\varepsilon_k} \textbf{e}_i
+ \sqrt{\varepsilon_k}^{-1} A_{i,n}  \textbf{e}_{n})^T,
$$
where $\textbf{0}$ is the zero vector.
Clearly, $A_k$ is a CP-completion of the matrix in (\ref{approximation}),
and the sequence $\{A_k|_E\}$ converges
to $\mathbf{a}$ as $k \rightarrow +\infty$.

Second, when two or more diagonal entries are missing,
the proof is same as in the above. We omit it here for cleanness.
\end{proof}

\br\label{noteonconvergence}
Theorem \ref{convergence} implies that
the set of all $CP$-completable $E$-matrices is not closed,
if some diagonal entries are missing.
\er

When an $E$-matrix has only one given diagonal entry,
there is a nice property of CP-completion.

\begin{proposition}
\label{onepartknow}
Let $A$ be a nonnegative $E$-matrix.
If only one diagonal entry of $A$ is given and it is positive,
then $A$ is CP-completable.
\end{proposition}

\begin{proof} Without loss of generality, we assume the first diagonal is given and positive.
Let $\tilde n$ be the number of the given entries in the first row.
So, $1\leq \tilde n\leq n$. If $\tilde n=1$, let
\begin{eqnarray*}
C = A_{11} \textbf{e}_1 \textbf{e}_1^T
+ \sum\limits_{2\leq i<j\leq n, (i,j)\in E} {A_{ij} (\textbf{e}_i + \textbf{e}_j)(\textbf{e}_i + \textbf{e}_j)^T}.
\end{eqnarray*}
Then $C$ is a CP-completion of $A$.
Otherwise, if $\tilde n>1$,   let
\begin{eqnarray*}
C &=& \sum\limits_{1=i<  j\leq n, (1,j)\in E}
({\sqrt{\frac{A_{11}}{\tilde n-1}} \textbf{e}_1 + \sqrt{\frac{\tilde n-1}{A_{11}}} A_{1j} \textbf{e}_j})
({\sqrt{\frac{A_{11}}{\tilde n-1}} \textbf{e}_1 + \sqrt{\frac{\tilde n-1}{A_{11}}} A_{1j} \textbf{e}_j})^T \\
&&\quad + \sum\limits_{2\leq i<j\leq n, (i,j)\in E} {A_{ij} (\textbf{e}_i + \textbf{e}_j)(\textbf{e}_i + \textbf{e}_j)^T}.
\end{eqnarray*}
It can be easily checked that $C$ is a CP-completion of $A$.
\end{proof}

\br\label{remark4}
1) If an $E$-matrix is not partially CP-completable, a certificate
(i.e., the relaxation (\ref{SDPR}) is infeasible)
can be obtained by applying Algorithm \ref{Algorithm} to
its maximum principle submatrix.
2) If an $E$-matrix is partially CP-completable,
then Algorithm \ref{Algorithm} can
give a CP-completion, up to an arbitrarily tiny positive perturbation
(applied to given diagonal entries).
\er

\section{Numerical experiments}

In this section, we present numerical experiments
for solving CP-completion problems by using Algorithm \ref{Algorithm}.
We use software GloptiPoly~3 \cite{HenrionJJ} and SeDuMi \cite{Sturm}
to solve semidefinite relaxations in (\ref{SDPR}).
We choose $d=4$ and $k=2$ in Step 0 of Algorithm \ref{Algorithm}.

\bex\label{Example1} \upshape
Consider the $E$-matrix $A$ given as (cf. \cite[Exercise 2.57]{BermanN}):
\begin{equation}\label{Exa1}
\left[\begin{array}{cccc}
  b&3 &0 &\ast\\
  3&6 &3 &0\\
  0&3 &6 &3\\
  \ast&0 &3 &b
\end{array}\right],
\end{equation}
where $b \geq 0$ is a parameter.
For a symmetric nonnegative matrix of order $n\leq 4$,
it is completely positive if and only if
it is positive semidefinite (cf. \cite{MaxfieldM}),
i.e., all its principal minors are nonnegative.
Let $c=A_{14}$, the missing value.
Then $A$ is completely positive if and only if
$$
b\geq 0, c\geq 0, 2b-3\geq 0, b-2\geq 0, 2b^2-3b-2c^2\geq 0, (b-2)^2-(c+1)^2\geq 0.
$$
The above is satisfiable if and only if $b\geq 3$, i.e.,
$A$ is CP-completable if and only if $b \geq 3$.
When $b=3$, $A$ is CP-completable only for $c=0$.

We choose $b=3$ and apply Algorithm \ref{Algorithm}.
It terminates at Step 3 with $k=3$,
and gives the CP-completion
\[
A=\left[\begin{array}{cccc}
  3&3 &0 & 0.0000\\
  3&6 &3 &0\\
  0&3 &6 &3\\
  0.0000&0 &3 &3
\end{array}\right] =\sum_{i=1}^{3} \rho_i u_i u_i^T,
\]
where $u_i$ and $\rho_i$ are given in Table \ref{ex1}.
\begin{table}[htbp]
\begin{center} \footnotesize
\renewcommand{\arraystretch}{1.2}
\begin{tabular}{|c|c|c|} \hline
$i$ &$u_i$ & $\rho_i$ \\ \hline
$1$ &$(0.5000, 0.5000, 0.0000, 0.0000)^T$ & 12.0000 \\ \hline
$2$ &$(0.0000, 0.0000, 0.5000, 0.5000)^T$ & 12.0000 \\ \hline
$3$ &$(0.0000, 0.5000, 0.5000, 0.0000)^T$ & 12.0000 \\ \hline
\end{tabular}
\end{center}
\caption{The points $u_i$ and weights $\rho_i$ in Example \ref{Example1}.}
\label{ex1}
\end{table}

\eex

\bex\label{Example2}  \upshape
Consider the $E$-matrix $A$ given as:
\begin{equation}\label{Exa2}
\left[\begin{array}{ccccc}
  \ast&4 &1 &2 &2\\
  4&\ast &0 &1 &3\\
  1&0 &\ast &1 &2\\
  2&1 &1 &\ast &1\\
  2&3 &2 &1 &\ast
\end{array}\right].
\end{equation}
All its diagonal entries are missing.
By Proposition \ref{noknow}, we know $A$ is CP-completable.
We apply Algorithm \ref{Algorithm}.
It terminates at Step 3 with $k=3$, and gives the CP-completion:
\[
 \left[\begin{array}{ccccc}
  5.8127&4 &1 &2 &2\\
  4&4.6697 &0 &1 &3\\
  1&0 &2.2682 &1 &2\\
  2&1 &1 &0.9087 &1\\
  2&3 &2 &1 &4.7740
\end{array}\right] =\sum_{i=1}^{3} \rho_i u_i u_i^T,
\]
where $u_i$ and $\rho_i$ are shown in Table \ref{ex2}.
\begin{table}[htbp]
\begin{center} \footnotesize
\renewcommand{\arraystretch}{1.2}
\begin{tabular}[t]{|c|c|c|} \hline
 $i$ &$u_i$ & $\rho_i$ \\ \hline
$1$ &$(0.1595, 0.0000, 0.3619, 0.1595, 0.3191)^T$ & 17.3224 \\ \hline
$2$ &$(0.1122, 0.4258, 0.0000, 0.0000, 0.4620)^T$ & 13.9667 \\ \hline
$3$ &$(0.4957, 0.3179, 0.0000, 0.1488, 0.0376)^T$ & 21.1443 \\ \hline
\end{tabular}
\end{center}
\caption{The points $u_i$ and weights $\rho_i$ in Example \ref{Example2}.}
\label{ex2}
\end{table}
The length of the CP-decomposition is 3,
which is much shorter than 9 given by (\ref{3.1}).
This shows an advantage of Algorithm~\ref{Algorithm}.
\eex

\bex\label{Example33} \upshape
Consider the $E$-matrix $A$ given as
\begin{equation}\label{Exa33}
\left[\begin{array}{ccccc}
  6.1232&4.1232 &1.1233 &2.1231 &2.3321\\
  4.1232&\ast &0 &1.0987 &3.2873\\
  1.1233&0 &3.2318 &1.2332 &2.1232\\
  2.1231&1.0987 &1.2332 &\ast &1.1232\\
  2.3321&3.2873 &2.1232 &1.1232 &\ast
\end{array}\right].
\end{equation}
By Proposition \ref{Kfull}, the space $\mathbb{R}[x]_{\mathcal {E}}$ is not $\Delta$-full.
We apply Algorithm \ref{Algorithm}.
It terminates at Step 3 with $k=4$, and gives the CP-completion:
\[
  \left[\begin{array}{ccccc}
  6.1232&4.1232 &1.1233 &2.1231 &2.3321\\
  4.1232&5.5494 &0 &1.0987 &3.2873\\
  1.1233&0 &3.2318 &1.2332 &2.1232\\
  2.1231&1.0987 &1.2332 &1.0430 &1.1232\\
  2.3321&3.2873 &2.1232 &1.1232 &3.6641
\end{array}\right] =\sum_{i=1}^{4} \rho_i u_i u_i^T,
\]
where $u_i$ and $\rho_i$ are shown in Table \ref{ex33}.
\begin{table}[htbp]
\begin{center} \footnotesize
\renewcommand{\arraystretch}{1.2}
\begin{tabular}{|c|c|c|} \hline
 $i$ &$u_i$ & $\rho_i$ \\ \hline
$1$ &$(0.3499, 0.3637, 0.0000, 0.0807, 0.2057)^T$ & 32.4007 \\ \hline
$2$ &$(0.0000, 0.5555, 0.0000, 0.0650, 0.3795)^T$ & 4.0941 \\ \hline
$3$ &$(0.4805, 0.0000, 0.2503, 0.2692, 0.0000)^T$ & 9.3406 \\ \hline
$4$ &$(0.0000, 0.0000, 0.4925, 0.1124, 0.3951)^T$ & 10.9106 \\ \hline
\end{tabular}
\end{center}
\caption{The points $u_i$ and weights $\rho_i$ in Example \ref{Example33}.}
\label{ex33}
\end{table}
This also shows a nice property of Algorithm \ref{Algorithm}:
if it exists, a CP-completion can be found,
even if $\mathbb{R}[x]_{\mathcal {E}}$ is not $\Delta$-full.
\eex

\bex\label{Example4}\upshape
Consider the $E$-matrix $A$ given as:
\begin{equation}\label{Exa4}
\left[\begin{array}{cccccc}
  \ast&7 &1 &3 &9 &\ast\\
  7&\ast &5 &8 &5 &3\\
  1&5 &\ast &2 &2 &\ast\\
  3&8 &2 &3 &1 &4\\
  9&5 &2 &1 &\ast &1\\
  \ast &3 &\ast &4 &1 &\ast
\end{array}\right].
\end{equation}
Only one diagonal entry is given.
By Proposition \ref{onepartknow}, (\ref{Exa4}) is CP-completable.
We apply Algorithm \ref{Algorithm}.
It terminates at Step 3 with $k=5$,
and gives the CP-completion
\[
  \left[\begin{array}{cccccc}
  11.3758&7 &1 &3 &9 &6.1225\\
  7&32.5203 &5 &8 &5 &3\\
  1&5 &4.2013 &2 &2 &2.5314\\
  3&8 &2 &3 &1 &4\\
  9&5 &2 &1 &11.1114 &1\\
  6.1225 &3 &2.5314 &4 &1 &10.8581
\end{array}\right] =\sum_{i=1}^{9} \rho_i u_i u_i^T,
\]
where $u_i$ and $\rho_i$ are listed in Table \ref{ex4}.
\begin{table}[htbp]
\begin{center} \footnotesize
\renewcommand{\arraystretch}{1.2}
\begin{tabular}{|c|c|c|} \hline
 $i$ &$u_i$ & $\rho_i$ \\ \hline
$1$ &$(0.0000, 0.0458, 0.5132, 0.0224, 0.4187, 0.0000)^T$ & 1.9950 \\ \hline
$2$ &$(0.0000, 0.4881, 0.2702, 0.1191, 0.1225, 0.0000)^T$ & 5.6983 \\ \hline
$3$ &$(0.0434, 0.5268, 0.1914, 0.1203, 0.1183, 0.0000)^T$ & 41.6816 \\ \hline
$4$ &$(0.0000, 0.0000, 0.3857, 0.1581, 0.0000, 0.4562)^T$ & 11.2576 \\ \hline
$5$ &$(0.1508, 0.5697, 0.0000, 0.1599, 0.0000, 0.1197)^T$ & 44.0288 \\ \hline
$6$ &$(0.1929, 0.5399, 0.0000, 0.1050, 0.1622, 0.0000)^T$ & 17.9222 \\ \hline
$7$ &$(0.2977, 0.0000, 0.0324, 0.1558, 0.0000, 0.5141)^T$ & 29.2892 \\ \hline
$8$ &$(0.4287, 0.0842, 0.0000, 0.0000, 0.4871, 0.0000)^T$ & 11.0745 \\ \hline
$9$ &$(0.4121, 0.0000, 0.0306, 0.0000, 0.4875, 0.0697)^T$ & 29.4268 \\ \hline
\end{tabular}
\end{center}
\caption{The points $u_i$ and weights $\rho_i$ in Example \ref{Example4}.}
\label{ex4}
\end{table}
\eex

\bex\label{Example5222}\upshape
Consider the $E$-matrix $A$ given as (cf. \cite[Example 1.35]{BermanN}):
\begin{equation}\label{Exa5}
\left[\begin{array}{ccccc}
  1&1 &\ast &\ast &0\\
  1&1 &1 &\ast &\ast\\
  \ast&1 &1 &1 &\ast\\
  \ast&\ast &1 &1 &1\\
 0&\ast &\ast &1 &1
\end{array}\right].
\end{equation}
It is shown in \cite{BermanN} that (\ref{Exa5}) is not CP-completable.
We apply Algorithm \ref{Algorithm} to verify this fact.
It terminates at Step 1 with $k=3$ as (\ref{SDPR}) is
infeasible, which confirms that (\ref{Exa5}) is not
CP-completable.
\eex

\bex\label{Example6}\upshape
Consider the $E$-matrix $A$ given as:
\begin{equation}\label{Exa6}
\left[\begin{array}{ccccc}
  1&1 &2 &\ast &4\\
  1&1 &3 &\ast &3\\
  2&3 &3 &3 &\ast\\
  \ast&\ast &3 &2 &\ast\\
  4&3 &\ast &\ast &\ast
\end{array}\right].
\end{equation}
By Proposition \ref{Kfull},
the set $\mathbb{R}[x]_{\mathcal {E}}$ is not $\Delta$-full.
We apply Algorithm \ref{Algorithm} to solve this CP-completion problem.
It terminates at Step 1 with $k=1$, because (\ref{SDPR}) is infeasible.
This shows that the $E$-matrix $A$ is not CP-completable.
By this example, we can see that
Algorithm~\ref{Algorithm} might get a certificate
for non-CP-completability, even if
$\mathbb{R}[x]_{\mathcal {E}}$ is not $\Delta$-full.
\eex

\bex\label{Exampleco7}\upshape
Consider the $E$-matrix $A$ given as in Remark~\ref{remark3}:
\begin{equation}\label{Exa7}
\left[\begin{array}{ccc}
  1&1 &2 \\
  1&1 &3 \\
  2&3 &\ast
\end{array}\right].
\end{equation}
We have already seen that $A$ is not CP-completable.
We apply positive perturbations to $A$ as follows:
\begin{equation}\label{pms}
\left[\begin{array}{ccc}
  1+10^{-l}&1 &2 \\
  1&1+10^{-l} &3 \\
  2&3 &\ast
\end{array}\right],\quad l=1,2,\cdots.
\end{equation}
By the proof of Theorem \ref{convergence}, we know that (\ref{pms}) is CP-completable for all $l$.
For $l=1, 2, 3$, Algorithm \ref{Algorithm} produces
the following CP-completions:
\[
A_{1}=\left[\begin{array}{ccc}
  1.1&1 &2 \\
  1&1.1 &3 \\
  2&3 &10.9524
\end{array}\right] =\sum_{i=1}^{2} \lambda_i u_i u_i^T,
\]

\[
A_{2}=\left[\begin{array}{ccc}
  1.01&1 &2 \\
  1&1.01 &3 \\
  2&3 &56.2189
\end{array}\right] =\sum_{i=1}^{2} \rho_i v_i v_i^T,
\]
and
\[
A_{3}=\left[\begin{array}{ccc}
  1.001&1 &2 \\
  1&1.001 &3 \\
  2&3 &487.2967
\end{array}\right] =\sum_{i=1}^{4} \sigma_i \omega_i \omega_i^T.
\]
The points and their weights are shown in Tables~\ref{ex71} and \ref{ex72}.

\begin{table}[htbp]
\begin{center} \footnotesize
\renewcommand{\arraystretch}{1.2}
\begin{tabular}{|c|c|c|c|c|} \hline
$i$ &$u_i$ &$\lambda_i$ &$v_i$ &$\rho_i$ \\ \hline
$1$ &$(0.1254, 0.1881, 0.6866)^T$ & 23.2350 & $(0.0327, 0.0490, 0.9183)^T$ & 66.6636\\ \hline
$2$ &$(0.6190, 0.3810, 0.0000)^T$ & 1.9174 & $(0.5124, 0.4876, 0.0000)^T$ & 3.5753 \\ \hline
\end{tabular}
\end{center}
\caption{The points and weights for $A_1$ and $A_2$ in Example \ref{Exampleco7}.}
\label{ex71}
\end{table}

\begin{table}[htbp]
\begin{center} \footnotesize
\renewcommand{\arraystretch}{1.2}
\begin{tabular}{|c|c|c|} \hline
 $i$ &$\omega_i$ & $\sigma_i$ \\ \hline
$1$ &$(0.5012, 0.4987, 0.0001)^T$ & 3.7527 \\ \hline
$2$ &$(0.0682, 0.0696, 0.8623)^T$ & 11.6502 \\ \hline
$3$ &$(0.0027, 0.0048, 0.9925)^T$ & 485.9000 \\ \hline
$4$ &$(0.4545, 0.3039, 0.2487)^T$ & 0.0023 \\ \hline
\end{tabular}
\end{center}
\caption{The   points and   weights for $A_3$ in Example \ref{Exampleco7}.}
\label{ex72}
\end{table}

When $l \geq 4$, the resulting semidefinite relaxations (\ref{SDPR})
are ill-conditioned, and the semidefinite programming solver {\tt SeDuMi}
has trouble to solve them accurately.
This is because this $E$-matrix is not CP-completable,
but it has an arbitrarily tiny perturbation
that is CP-completable.
\eex

%
%
%

\section{Conclusions and Discussions}

This paper proposes a semidefinite algorithm
(i.e., Algorithm \ref{Algorithm}) for solving general CP-completion problems.
When all the diagonal entries are given,
Algorithm \ref{Algorithm} can give a certificate for non-CP-completability.
If a partial matrix is CP-completable,
Algorithm \ref{Algorithm} almost always gets
a CP-completion, as well as a CP-decomposition.
When some diagonal entries are missing,
if the maximum principal submatrix is not CP-completable,
then a certificate for non-CP-completability can be obtained; if it is CP-completable,
then Algorithm \ref{Algorithm} also almost always gives a CP-completion.

CP-completion has wide applications (cf.~\cite{BermanN}).
Here we show one in probability theory.
Let $x$ be a random vector in $\mathbb{R}^n$.
Suppose its expectation $\mathbb{E} x = b$
and partial entries of its covariance matrix are known, say,
for an index set $E$ the entries
\[
X_{ij} = \mathbb{E}[( x_i - b_i ) ( x_j - b_j )]
\]
with $(i,j) \in E$ are known.
We want to investigate for what values of $X_{ij} \,((i,j) \in E)$
the density function of $x$ is supported in the nonnegative orthant
$\mathbb{R}_+^n$. This question is basic and natural,
because many statistical quantities are positive in the world.
Interestingly, this question
can be formulated as a CP-completion problem.
From the expression of $X_{ij}$, we can see that
\[
\mathbb{E}( x_i x_j) =
X_{ij}  + b_i b_j.
\]
Let $A$ be the $E$-matrix such that
$A_{ij} = X_{ij}  + b_ib_j$
for all $(i,j) \in E$. Whether the random vector $x$
has a density function supported
in $\mathbb{R}_+^n$ or not is basically equivalent to
whether the following partial matrix
\[
\begin{bmatrix}
\ast &  b^T \\
b &  A
\end{bmatrix}
\]
is CP-completable or not.
In the above, $A$ is also a partial matrix;
only the entries $A_{ij}$ with $(i,j) \in E$ are known.

CP-completion also has applications in nonconvex quadratic optimization.
Let $E$ be an index set.  A typical quadratic optimization problem is
\begin{equation} \label{qp:E}
\left\{ \begin{array}{rl}
\min & \sum_{(i,j) \in E} a_{ij} x_i x_j \\
s.t. & p(x) = 1, \, x \geq 0,
\end{array} \right.
\end{equation}
where $p(x) = \sum_{(i,j) \in E} p_{ij} x_i x_j $ is a given polynomial.
To solve this nonconvex optimization problem globally, we need to
characterize the cone
\[
\mathcal{C}_E  = \{
y \in \mathbb{R}^{E} : \,
 y = C|_E \mbox{ for some CP-matrix } C \in \mathbb{R}^{n\times n}
\}.
\]
It can be shown that (\ref{qp:E}) is equivalent to
the linear convex problem
\begin{equation} \label{qp:CP}
\left\{ \begin{array}{rl}
\min & \sum_{(i,j) \in E} a_{ij} y_{ij} \\
s.t. &  \sum_{(i,j) \in E} p_{ij} y_{ij} = 1, \,  y \in \mathcal{C}_E .
\end{array} \right.
\end{equation}
To design efficient numerical methods for solving (\ref{qp:CP}),
we need to check the memberships in $\mathcal{C}_E$.
This is equivalent to solving CP-completion problems.

In this paper, we mainly focus on determining
whether a partial matrix is CP-completable or not.
However, we did not discuss the question
of how to get a CP-completion whose CP-rank is minimum.
This question is hard, and there exists few work
about it, to the best knowledge of the authors.
This is an interesting future work.

%

\end{document}